\documentclass[10pt]{article}
\usepackage{amsfonts}
\usepackage{amssymb}
\usepackage{graphicx}
\usepackage{amsmath,amsthm}
\usepackage{url,paralist}
\usepackage{anysize}
\usepackage{hyperref}

\newtheorem{theorem}{Theorem}[section]
\newtheorem{lemma}[theorem]{Lemma}
\newtheorem{proposition}[theorem]{Proposition}

\theoremstyle{definition}

\newtheorem{remark*}[theorem]{Remark}

\newcommand\R{\mathbb{R}}
\newcommand\Z{\mathbb{Z}}
\newcommand\lra{\longrightarrow}

\begin{document} 

\title{{\huge Tetrahedra on deformed spheres}\\
{\huge and integral group cohomology}}
\author{Pavle V. M. Blagojevi\'c%
\thanks{Supported by the grant 144018 of the Serbian Ministry of Science and
Technological development} \\ 
Mathemati\v cki Institut\\
Knez Michailova 35/1\\
11001 Beograd, Serbia\\
\url{pavleb@mi.sanu.ac.yu} \and \setcounter{footnote}{6} G\"unter M. Ziegler%
\thanks{Partially supported by the German Research Foundation DFG} \\ 
Inst.\ Mathematics, MA 6-2\\
TU Berlin\\
D-10623 Berlin, Germany\\
\url{ziegler@math.tu-berlin.de}}
\date{{\small August 28, 2008}}
\maketitle

\begin{abstract}
\noindent We show that for every injective continuous map
$f:S^2\rightarrow\mathbb{R}^3$ 
there are four distinct points in the image of~$f$ such that
the convex hull is a tetrahedron with the property that two opposite edges
have the same length and the other four edges are also of equal length. This
result represents a partial result for the topological Borsuk problem for
$\mathbb{R}^3$. Our proof of the geometrical claim, via Fadell--Husseini
index theory, provides an instance where arguments based on group cohomology
with integer coefficients yield results that cannot be accessed using only
field coefficients.
\end{abstract}

\section{Introduction}

The motivation for the study of the existence of particular types of
tetrahedra on deformed $2$-spheres is twofold. The topological Borsuk
problem, as considered in \cite{S}, along with the square peg problem \cite{Shn}
inspire the search for possible polytopes with nice metric properties
whose vertices lie on the continuous images of spheres. Beyond their
intrinsic interest, these problems can be used as testing grounds for tools
from equivariant topology, e.g.\ for comparing the strength of
Fadell--Husseini index theory with ring resp.\ field coefficients.

\begin{figure}[tbh]
\centering\includegraphics[scale=0.55]{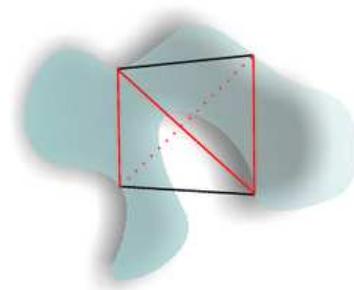}
\caption{$D_8$-invariant tetrahedra on deformed sphere $S^2$}
\end{figure}

The following theorem will be proved through the use of Fadell--Husseini
index theory with coefficients in the ring $\mathbb{Z}$. It is also going to
be demonstrated that Fadell--Husseini index theory with coefficients in
field $\mathbb{F}_2$ has no power in this instance (Section \ref{Sec:Final-1}).

\begin{theorem}
\label{Th:Main-1} Let $f:S^2\rightarrow\mathbb{R}^3$ an injective continuous
map. Then its image contains vertices of a tetrahedron that has the symmetry
group $D_8$ of a square. That is, there are four distinct points 
$\xi_1$, $\xi_2$, $\xi_3$ and $\xi_4$ on $S^2$ such that
\begin{equation*}
d(f(\xi_1),f(\xi_2)) = d(f(\xi_2),f(\xi_3)) = d(f(\xi_3),f(\xi_4)) =
d(f(\xi_4),f(\xi_1))
\end{equation*}
and
\begin{equation*}
d(f(\xi_1),f(\xi_3)) = d(f(\xi_2),f(\xi_4)).
\end{equation*}
\end{theorem}

\begin{remark*}
The proof is not going to use any properties of $\mathbb{R}^3$ except that
it is a metric space. Thus in the statement of the theorem, $\mathbb{R}^3$
can be replaced by any metric space $(M,d)$.
\end{remark*}

\begin{remark*}
Unfortunately, the methods used for the proof of Theorem \ref{Th:Main-1} do
not imply any conclusion when applied to the square peg problem (see Section 
\ref{Sec:Final-2}). On the other hand, if the square peg problem could be
solved for the continuous Jordan curves, then it would imply the result of
Theorem \ref{Th:Main-1}.
\end{remark*}

\section{\label{Sec:CSTM}Introducing the equivariant question}

Let $f:S^2\rightarrow\mathbb{R}^3$ be an injective continuous map. Denote by
$D_8$ the symmetry group of a square, that is, the $8$-element dihedral
group $D_8 = \langle \omega ,j~|~\omega^4=j^2=1,~\omega j=j\omega ^3\rangle $.

\subsubsection*{A few $D_8$-representations.}

The vector spaces
\begin{eqnarray*}
U_4 &=&\{(x_1,x_2,x_3,x_4)\in\mathbb{R}^4~|~x_1+x_2+x_3+x_4=0\}, \\
U_2 &=&\{(x_1,x_2)\in\mathbb{R}^2~|~x_1+x_2=0\}
\end{eqnarray*}
are $D_8$-representations with actions given by

(a) for $(x_1,x_2,x_3,x_4)\in U_4$:
\begin{equation*}
\omega \cdot (x_1,x_2,x_3,x_4)=(x_2,x_3,x_4,x_1), \qquad j\cdot
(x_1,x_2,x_3,x_4)=(x_3,x_2,x_1,x_4),
\end{equation*}

(b) for $(x_1,x_2)\in U_2:$
\begin{equation*}
\omega \cdot (x_1,x_2)=(x_2,x_1), \qquad j\cdot (x_1,x_2)=(x_2,x_1),
\end{equation*}

\subsubsection*{The configuration space.}

Let $X=S^2\times S^2\times S^2\times S^2$ and let $Y$ be the subspace given
by
\begin{equation*}
Y=\left\{ (x,y,x,y)~|~x,y\in S^2\right\} \approx S^2\times S^2.
\end{equation*}
The configuration space to be considered is the space
\begin{equation*}
\Omega :=X\backslash Y.
\end{equation*}

\noindent Let a $D_8$-action on $X$ be induced by
\begin{equation}
\begin{array}{ccc}
\omega \cdot (\xi_1,\xi_2,\xi_3,\xi_4)=(\xi_2,\xi_3,\xi _4,\xi_1), &  &
j\cdot (\xi_1,\xi_2,\xi_3,\xi_4)=(\xi _4,\xi_3,\xi_2,\xi_1),
\end{array}
\notag
\end{equation}
for $(\xi_1,\xi_2,\xi_3,\xi_4)\in X$.

\subsubsection*{A test map.}

Let $\tau :\Omega \rightarrow U_{4}\times U_{2}$ be a map defined for $(\xi
_{1},\xi _{2},\xi _{3},\xi _{4})\in X$ by
\begin{equation}
\tau (\xi _{1},\xi _{2},\xi _{3},\xi _{4})=
( d_{12}-\tfrac\Delta4,d_{23}-\tfrac\Delta4,
  d_{34}-\tfrac\Delta4,d_{41}-\tfrac\Delta4 ) \times 
( d_{13}-\tfrac\Phi2,  d_{24}-\tfrac\Phi2   )   \label{Def-t-1}
\end{equation}
where $d_{ij}:=d(f\left( \xi _{i}\right) ,f\left( \xi _{j}\right) )$ and
\begin{equation*}
\Delta =d_{12}+d_{23}+d_{34}+d_{14},~~\qquad \Phi =d_{13}+d_{24}.
\end{equation*}
With the $D_{8}$-actions introduced above the test map $\tau $ is
$D_{8}$-equivariant. The following proposition connects our set-up 
with the tetrahedron problem.

\begin{proposition}
\label{prop:testmap} If there is no $D_8$-equivariant map
\begin{equation}
\alpha :\Omega \rightarrow (U_4\times U_2)\backslash 
         (\{\mathbf{0}\}\times \{\mathbf{0}\})  \label{map-1}
\end{equation}
then Theorem \ref{Th:Main-1} follows.
\end{proposition}

\begin{proof}
If there is no $D_8$-equivariant map
$\Omega \rightarrow (U_4\times U_2)\backslash 
                    (\{\mathbf0\}\times \{\mathbf0\})$,
then for every continuous embedding $f:S^2\rightarrow\R^3$ there is a point
$\xi =(\xi_1,\xi_2,\xi_3,\xi_4)\in \Omega
=X\backslash Y$ such that
\begin{equation}
\tau (\xi_1,\xi_2,\xi_3,\xi_4)= (\mathbf0,\mathbf0)\in U_4\times U_2. 
    \label{a-2}
\end{equation}
From (\ref{a-2}) we conclude that
\begin{equation}
d_{12}=d_{23}=d_{34}=d_{14}=\tfrac{\Delta }{4}\text{ \quad and \quad }
d_{13}=d_{24}=\tfrac{\Phi }{2}.  \label{a-3}
\end{equation}
It only remains to prove that all four points are different. Since
$(\xi_1,\xi_2,\xi_3,\xi_4)\notin Y$ we have $\xi_1\neq \xi_3$ or
$\xi_2\neq \xi_4$. By symmetry we may assume that $\xi_1\neq \xi_3$.
The map $f$ is injective, therefore $f(\xi_1)\neq f(\xi_3)$ and
consequently $d_{13}\neq 0$. Now
\[
d_{13}\neq 0
\ \ \Rightarrow\ \
d_{24}\neq 0
\ \ \Rightarrow\ \
f(\xi_1)\neq f(\xi_3),\ f(\xi_2)\neq f(\xi_4)
\ \ \Rightarrow\ \
\xi_1\neq \xi_3,\ \xi_2\neq \xi_4.
\]
Let us assume, without loss of generality, that $\xi_1=\xi_2$.
Then $d_{12}=d_{23}=d_{34}=d_{14}=0$, which implies that $d_{13}\leq
d_{12}+d_{23}=0$. This yield a contradiction to $d_{13}\neq 0$. 
Thus $\xi_1\neq \xi_2$.
\end{proof}

By Proposition \ref{prop:testmap}, Theorem \ref{Th:Main-1} is a consequence
of the following topological result.

\begin{theorem}
\label{Th:Main-2}There is no $D_8$-equivariant map $\Omega \rightarrow
S(U_4\times U_2)$.
\end{theorem}

\section{Proof of Theorem \protect\ref{Th:Main-2}}

The proof is going to be conducted through a comparison of the Serre
spectral sequences with $\mathbb{Z}$-coefficients of the Borel constructions
associated with the spaces $\Omega $ and $S(U_{4}\times U_{2})$ and the
subgroup $\mathbb{Z}_{4}=\langle \omega \rangle $ of $D_{8}$. In other
words, we determine the $\mathbb{Z}_{4}$ Fadell--Husseini index of these
spaces living in 
$H_{4}^{\ast }(\mathbb{Z}_{4};\mathbb{Z})=\mathbb{Z}[U]/4U$, $\deg U=2$.

The Fadell--Husseini index of a $G$-space $X$ is the kernel of the map 
$\pi_{X}^{\ast }:H^{\ast }(\mathrm{B}G,\mathbb{Z})\rightarrow 
                 H^{\ast }(X\times_{G}\mathrm{E}G,\mathbb{Z)}$ 
induced by the projection 
$\pi _{X}:X\times _{G}\mathrm{E}G\rightarrow \mathrm{B}G$. 
If $E_{\ast }^{\ast ,\ast }$ denotes
the Serre spectral sequence of the Borel construction of $X$, then the
homomorphism $\pi _{X}^{\ast }$ can be presented as the composition%
\begin{equation}
H^{\ast }(\mathrm{B}G,\mathbb{Z})\rightarrow E_{2}^{\ast ,0}\rightarrow
E_{3}^{\ast ,0}\rightarrow E_{4}^{\ast ,0}\rightarrow ...\rightarrow
E_{\infty }^{\ast ,0}\subseteq H^{\ast }(X\times _{G}\mathrm{E}G,\mathbb{Z}).
\label{eq-1}
\end{equation}
Since the $E_{2}$-term of the spectral sequence is given by
$E_{2}^{p,q}=H^{p}(\mathrm{B}G,H^{q}(X,\mathbb{Z}))$ the first step
in the computation of the index is study of the cohomology
$H^{\ast }(X,\mathbb{Z})$ as a $G$-module (Section \ref{Sec1}).
The final step is explicit description of non-zero differentials
in the spectral sequence and application of the
presentation (\ref{eq-1}) of the homomorphism $\pi _{X}^{\ast }$ 
(Section \ref{Sec:SS}).

\subsection{The Index of $S(U_{4}\times U_{2})$}

Let $V^1$ be the $1$-dimensional complex representation of $\mathbb{Z}_4$
induced by $1\mapsto e^{i\pi/2}$. Then the representation
$U_4\subset\mathbb{R}^4$ seen as a $\mathbb{Z}_4$-representation 
decomposes into a sum of two
irreducible $\mathbb{Z}_4$-representations
\begin{equation*}
U_4=\mathrm{span}\{(1,0,-1,0),(0,1,0,-1)\} \oplus 
    \mathrm{span}\{(1,-1,1,-1)\}                  \cong V^1\oplus U_2.
\end{equation*}
Here \textquotedblleft $\mathrm{span}$\textquotedblright\ stands for all
$\mathbb{R}$-linear combinations of the given vectors. It can be also seen
that
\begin{equation*}
U_4\times U_2\cong V^1\oplus U_2\oplus U_2\cong V^1\oplus (V^1\otimes V^1).
\end{equation*}
Following \cite[Section 8, p.~271 and Appendix, page 285]{Aty} we deduce the
total Chern class of the $\mathbb{Z}_4$-representation $U_4\times U_2$
\begin{equation*}
c(U_4\times U_2)=c(V^1)\cdot c(V^1\otimes V^1)
\end{equation*}
and consequently the top Chern class
\begin{equation*}
c_2(U_4\times U_2)=c_1(V^1)\cdot c_1(V^1\otimes V^1)=c_1(V^1)\cdot
(c_1(V^1)+c_1(V^1))=2U^2\in H^*(\mathbb{Z}_4;\mathbb{Z}).
\end{equation*}
The $\mathbb{Z}_4$-index of the sphere $S(U_4\times U_2)$ is given by
\cite[Proposition 3.11]{B-Z} as
\begin{equation}
\mathrm{Index}_{\mathbb{Z}_4,\mathbb{Z}}S(U_4\times U_2)=\langle 2U^2\rangle.  
\label{index-1}
\end{equation}

\subsection{\label{Sec1}The cohomology $H^{\ast }(\Omega ;\mathbb{Z})$ 
                        as a $\mathbb{Z}_{4}$-module}

The cohomology is going to be determined via Poincar\'{e}--Lefschetz duality
and an explicit study of cell structures for the spaces $X$ and $Y$.

Poincar\'{e}--Lefschetz duality \cite[Theorem 70.2, page 415]{Mun} implies
that
\begin{equation}
H^*(\Omega;\mathbb{Z})=H^*(X\backslash Y;\mathbb{Z})\cong 
                       H_{8-\ast }(X,Y;\mathbb{Z})  \label{iso-1}
\end{equation}
and therefore we analyze the homology of the pair $(X,Y)$.

The long exact sequence in homology of the pair $(X,Y)$ yields that the
possibly non-zero homology groups of the pair $(X,Y)$ with
$\mathbb{Z}$-coefficients are
\begin{equation*}
H_{i}(X,Y;\mathbb{Z})=\left\{
\begin{array}{lll}
\mathbb{Z}[\mathbb{Z}_4]/\mathrm{im}\Phi , &  & i=2 \\
\ker \Phi , &  & i=3 \\
\mathbb{Z}[\mathbb{Z}_4]\oplus
\mathbb{Z}[\mathbb{Z}_4/\mathbb{Z}_2]/\mathrm{im}\Psi , &  & i=4 \\
\ker \Psi , &  & i=5 \\
\mathbb{Z}[\mathbb{Z}_4], &  & i=6 \\
\mathbb{Z}, &  & i=8
\end{array}
\right.
\end{equation*}
Thus explicit formulas for the maps $\Phi :H_2(Y;\mathbb{Z})\rightarrow
H_2(X;\mathbb{Z})$ and $\Psi :H_4(Y;\mathbb{Z})\rightarrow H_4(X;\mathbb{Z})$,
induced by the inclusion $Y\subset X$, are needed in order to determine
the homology $H_{\ast }(X,Y;\mathbb{Z})$ and its exact $\mathbb{Z}_4$-module
structure.

Let $x_1,x_2,x_3,x_4\in H_2(X;\mathbb{Z})$ be generators carried by
individual copies of $S^2$ in the product $X=S^2\times S^2\times S^2\times
S^2$. The generator of the group $\mathbb{Z}_4=\langle \omega \rangle $ acts
on this basis of $H_2(X;\mathbb{Z})$ by $\omega \cdot x_{i}=x_{i+1}$ where
$x_{5}=x_1$. Then by $x_{i}x_{j}\in H_4(X;\mathbb{Z})$, $i\neq j$, we denote
the generator carried by the product of $i$-th and $j$-th copy of $S^2$ in $X$. 
Since $\omega $ is not changing the orientation the action on 
$H_4(X;\mathbb{Z})$ is described by
\begin{equation*}
x_1x_2\overset{\cdot \omega } \longmapsto 
x_2x_3\overset{\cdot \omega } \longmapsto 
x_3x_4\overset{\cdot \omega } \longmapsto  
x_1x_4\text{ \qquad and \qquad }
x_1x_3\overset{\cdot \omega }{\longmapsto }x_2x_4.
\end{equation*}
Let similarly $y_1,y_2\in H_2(X;\mathbb{Z})$ be generators carried by
individual copies of $S^2$ in the product $Y=S^2\times S^2$. Then $\omega
\cdot y_1=y_2$ and $\omega \cdot y_2=y_1$. Again $y_1y_2$ denotes the
generator of $H_4(Y;\mathbb{Z})$ and $\omega \cdot y_1y_2=y_1y_2$. Note that
$\omega $ preserves the orientations of $X$ and $Y$ and therefore acts
trivially on $H_8(X;\mathbb{Z})$ and on $H_4(Y;\mathbb{Z})$.

\medskip

\noindent The inclusion $Y\subset X$ induces a map in homology 
$H_{\ast }(X;\mathbb{Z})\subset H_{\ast }(Y;\mathbb{Z})$, 
which in dimensions $2$ and $4$ is given by
\begin{equation*}
\begin{tabular}{cc}
$y_1\longmapsto x_1+x_3,$ & $y_2\longmapsto x_2+x_4,$ \\
&  \\
\multicolumn{2}{c}{$y_1y_2\longmapsto x_1x_2+x_2x_3+x_3x_4+x_1x_{4.}$}
\end{tabular}
\end{equation*}
This can be seen from the dual cohomology picture: An element is mapped to a
sum of generators intersecting its image, with appropriately attached
intersection numbers.

\noindent Thus $\Phi $ and $\Psi $ are injective and
\begin{equation*}
\begin{array}{ccc}
\mathrm{im}\Phi =\langle x_1+x_3,x_2+x_4\rangle , &  & \mathrm{im}\Psi
=\langle x_1x_2+x_2x_3+x_3x_4+x_1x_4\rangle .
\end{array}
\end{equation*}
Let $N=\mathbb{Z}\oplus\mathbb{Z}$ be the $\mathbb{Z}_4$-representation
given by $\omega \cdot (a,b)=(b,-a)$, while $M$ denotes the representation 
$\mathbb{Z}[\mathbb{Z}_4]/_{(1+\omega +\omega ^2+\omega ^3)\mathbb{Z}}$. Then
the non-trivial cohomology of the space $X\backslash Y$, as a 
$\mathbb{Z}_4$-module via the isomorphism (\ref{iso-1}), is given by
\begin{equation}
H^{i}(\Omega;\mathbb{Z})=\left\{
\begin{array}{lll}
N, &  & i=6 \\
M{\ \oplus }\mathbb{Z}[\mathbb{Z}_4/\mathbb{Z}_2], &  & i=4 \\
\mathbb{Z}[\mathbb{Z}_4], &  & i=2 \\
\mathbb{Z}, &  & i=0
\end{array}
\right.  \label{iso-2}
\end{equation}

\subsection{\label{Sec:SS} The Serre spectral sequence of the Borel
construction $\Omega \times_{\mathbb{Z}_4}\mathrm{E}\mathbb{Z}_4$}

The Serre spectral sequence associated to the fibration $\Omega \rightarrow
\Omega \times _{\mathbb{Z}_{4}}\mathrm{E}\mathbb{Z}_{4}\rightarrow 
\mathrm{B}\mathbb{Z}_{4}$ is a spectral sequence with non-trivial 
local coefficients,
since $\pi _{1}(\mathrm{B}\mathbb{Z}_{4})=\mathbb{Z}_{4}$ acts non-trivially
(\ref{iso-2}) on the cohomology $H^{\ast }(\Omega ;\mathbb{Z})$. The first
step in the study of such a spectral sequence is to understand the $H^{\ast
}(\mathbb{Z}_{4};\mathbb{Z})$-module structure on the rows of its $E_{2}$-term.

\noindent The $E_2$-term of the sequence is given by
\begin{equation*}
E_2^{p,q}=\left\{
\begin{array}{lll}
H^p(\mathbb{Z}_4,N), &  & q=6 \\
H^p(\mathbb{Z}_4,M) \oplus
H^p(\mathbb{Z}_4;\mathbb{Z}[\mathbb{Z}_4/\mathbb{Z}_2]), &  & q=4 \\
H^p(\mathbb{Z}_4;\mathbb{Z}[\mathbb{Z}_4]), &  & q=2 \\
H^p(\mathbb{Z}_4;\mathbb{Z}), &  & q=0 \\
0, &  & \text{otherwise.}
\end{array}
\right.
\end{equation*}

\begin{lemma}
$H^{p}(\mathbb{Z}_4;\mathbb{Z}[\mathbb{Z}_4]) = \left\{
\begin{array}{lll}
\mathbb{Z}, &  & p=0 \\
0, &  & p>0
\end{array}
\right. $
\newline
and multiplication by $U\in H^{p}(\mathbb{Z}_4;\mathbb{Z})$ is trivial, 
$U\cdot H^{p}(\mathbb{Z}_4;\mathbb{Z}[\mathbb{Z}_4])=0$.
\end{lemma}

For the proof one can consult \cite[Exercise 2, page 58]{Brown}.

\begin{lemma}
$H^*(\mathbb{Z}_4;\mathbb{Z}[\mathbb{Z}_4/\mathbb{Z}_2])\cong 
 H^*(\mathbb{Z}_2;\mathbb{Z})$, 
where the module structure is given by the restriction homomorphism 
$\mathrm{res}_{\mathbb{Z}_2}^{\mathbb{Z}_4}:
 H^*(\mathbb{Z}_4;\mathbb{Z})\rightarrow H^*(\mathbb{Z}_2;\mathbb{Z})$.

\noindent In other words, if we denote 
$H^*(\mathbb{Z}_2;\mathbb{Z})=\mathbb{Z}[ T]/2T$, $\deg T=2$, then
$\mathrm{res}_{\mathbb{Z}_2}^{\mathbb{Z}_4}(U)=T $ and consequently:

\begin{compactenum}[\rm(A)]
\item $H^*(\Z_4;\Z[\Z_4/\Z_2])$ is generated by one element of degree $0$
as a $H^*(\Z_4;\Z)$-module, and
\item multiplication by $U$ in $H^*(\Z_4;\Z[\Z_4/\Z_2])$ is an isomorphism,
while multiplication by $2U$ is zero.
\end{compactenum}
\end{lemma}

The proof is a direct application of Shapiro's lemma 
\cite[(6.3), page 73]{Brown} and a small part of the 
restriction diagram \cite[Section 4.5.2]{B-Z}.

\begin{lemma}
Let $\Lambda \in $ $H^*(\mathbb{Z}_4,M)$ denote an element of degree $1$
such that $4\Lambda =0$.\newline
Then $H^*(\mathbb{Z}_4,M)\cong H^*(\mathbb{Z}_4;\mathbb{Z})\cdot \Lambda $
as an $H^*(\mathbb{Z}_4;\mathbb{Z})$-module.
\end{lemma}

\begin{proof}
The short exact sequence of $\Z_4$-modules
\begin{equation*}
0\lra\Z\overset{1+\omega +\omega ^2+\omega ^3}{\lra }\Z[\Z_4]\lra M\lra 0
\end{equation*}
induces a long exact sequence in cohomology 
\cite[Proposition 6.1, page 71]{Brown},
which is natural with respect to $H^*(\Z_4;\Z)$-module multiplication.
Since $\Z[\Z_4]$ is a free module we get enough zeros to recover 
the information we need:
\[
\begin{array}{c@{~}c@{~}c@{~}c@{~}c@{~}c@{~}c@{~}c@{~}c@{~}c@{~}c@{~}c}
0 \lra& H^0(\Z_4;\Z)&{\lra}& H^0(\Z_4;\Z[\Z_4]) &{\lra}& H^0(\Z_4,M) &{\lra}& H^1(\Z_4;\Z) &{\lra}\\
       &     \Z       &      &          \Z        &      &             &      &     0               \\[3mm]
       &              &{\lra}& H^1(\Z_4;\Z[\Z_4]) &{\lra}& H^1(\Z_4,M) &{\lra}& H^2(\Z_4;\Z) &{\lra}\\
       &              &      &           0        &      &             &      &    \Z_4             \\[3mm]
       &              &{\lra}& H^2(\Z_4;\Z[\Z_4]) &{\lra}& \ldots  \\
       &              &      &           0        &                \\[4mm]

       &  \ldots      &{\lra}& H^i(\Z_4;\Z[\Z_4]) &{\lra}& H^i(\Z_4,M) &{\lra}& H^{i+1}(\Z_4;\Z) & {\lra}& H^{i+1}(\Z_4;\Z[\Z_4]) &{\lra}~\ldots\\
       &              &      &           0        &      &             &      &                  &       &  0   \\
\end{array}
\]
\end{proof}

\begin{lemma}
Let $\Upsilon \in H^*(\mathbb{Z}_4,N)$ denote an element of degree $1$ such
that $2\Upsilon =0$.\newline
Then $H^*(\mathbb{Z}_4,N)\cong 
H^*(\mathbb{Z}_4;\mathbb{Z}[\mathbb{Z}_4/\mathbb{Z}_2])\cdot \Upsilon $ 
as an $H^*(\mathbb{Z}_4;\mathbb{Z})$-module.
\end{lemma}

\begin{proof}
There is a short exact sequence of $\Z_4$-modules
\begin{equation*}
0\rightarrow N\overset{\alpha }{\rightarrow }\Z[\Z_4]\rightarrow L\rightarrow 0
\end{equation*}
where $L=\Z[\Z_4]/N$ and $\alpha (p,q)=(p,q,-p-q)$. 
The map $\alpha $ is well defined because the following diagram commutes
\begin{equation*}
\begin{array}{ccccc}
N=_{ab}\Z\oplus\Z \ni &(p,q)&\overset{\alpha }{\lra }& (p,q,-p,-q)&
\in\Z[\Z_4]\\
& {\downarrow } {\cdot \omega } &  & {\downarrow } {\cdot \omega } &  \\
N=_{ab}\Z\oplus\Z\ni &(q,-p)& \overset{\alpha }{\lra } & (q,-p,-q,p) & 
\in \Z[\Z_4]
\end{array}
\end{equation*}
The long exact sequence in group cohomology 
\cite[Prop.~6.1, p~71]{Brown} implies the result.
\end{proof}

\medskip

\noindent The $E_2$-term of the Borel construction $\left( X\backslash
Y\right) \times_{\mathbb{Z}_4}\mathrm{E}\mathbb{Z}_4$, 
with the $H^*(\mathbb{Z}_4;\mathbb{Z})$-module structure, 
is presented in Figure \ref{Fig-a1}.

\begin{figure}[tbh]
\centering\includegraphics[scale=0.60]{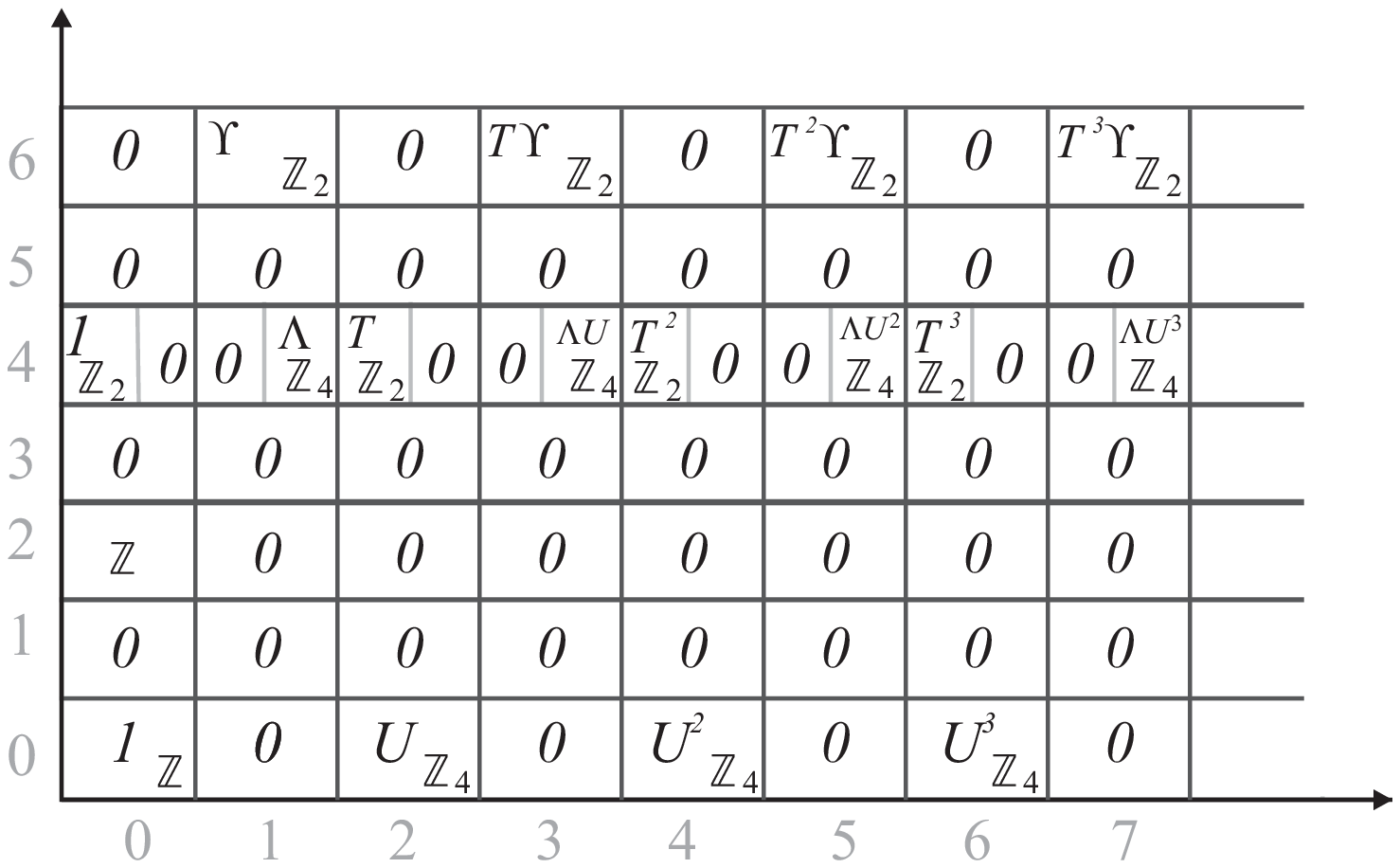}
\caption{The $E_2$-term}
\label{Fig-a1}
\end{figure}
The differentials of the spectral sequence are retrieved from the fact that
the $\mathbb{Z}_4$ action on $\Omega $ is free. Therefore 
$H_{\mathbb{Z}_4}^{i}(\Omega;\mathbb{Z})=0$ for all $i>8$. 
Since the spectral sequence is
converging to the graded group associated with 
$H_{\mathbb{Z}_4}^{i}(\Omega;\mathbb{Z})$ this means that 
for $p+q>8$ nothing survives. Thus the only
non-zero second differentials are 
$d_2:E_2^{2i+1,6}\rightarrow E_2^{2i+4,4}$, $d_2(T^{i}\Upsilon )=T^{i+1}$, 
$i>0$, as displayed in Figure \ref{Fig-ab}.

\begin{figure}[tbh]
\centering\includegraphics[scale=0.85]{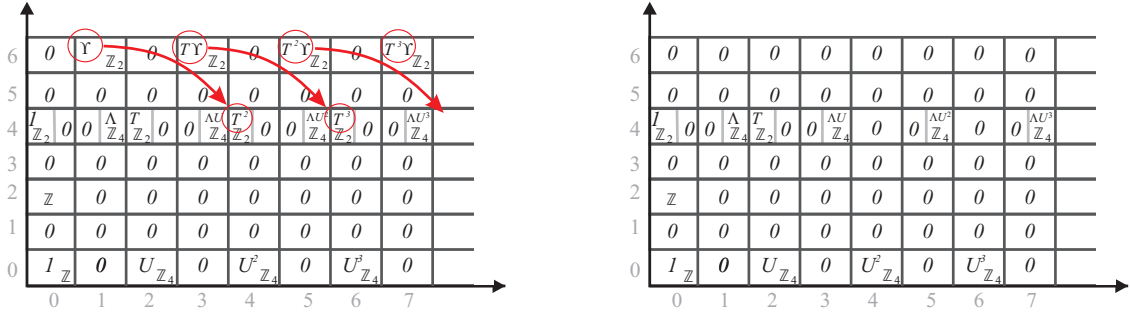}
\caption{Differentials in $E_2$ and $E_3$-terms}
\label{Fig-ab}
\end{figure}

\noindent The last remaining non-zero differentials are 
$d_4:E_4^{2i+1,4}\rightarrow E_4^{2i+6,0}$, $d_{6}(U^{i}\Lambda)=U^{i+3}$, 
$i>0$. Then $E_{5}=E_{\infty }$, \textit{cf.} Figure~\ref{Fig-ac}.

\begin{figure}[tbh]
\centering\includegraphics[scale=0.85]{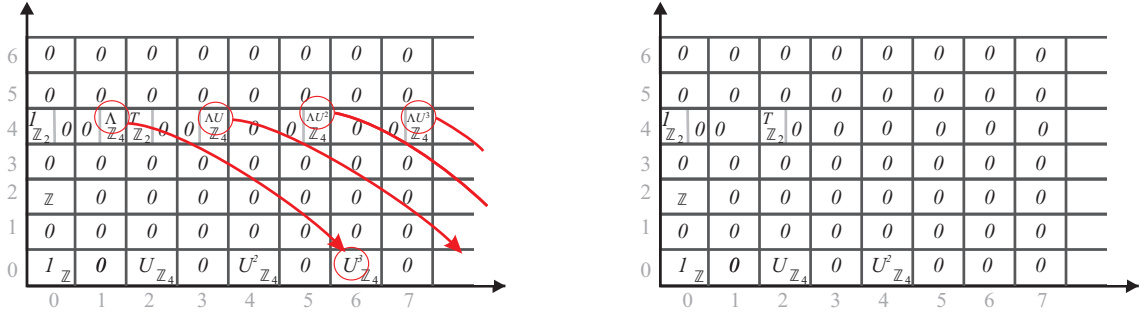}
\caption{Differentials in $E_4$ and $E_{5}$-terms}
\label{Fig-ac}
\end{figure}

\subsection{The index of $\Omega $}

The conclusion $d_{6}(\Lambda )=U^3$ implies that
\begin{equation*}
\mathrm{Index}_{\mathbb{Z}_4,\mathbb{Z}}\Omega =\langle U^3\rangle .
\end{equation*}
Since the generator $2U^2$ of the 
$\mathrm{Index}_{\mathbb{Z}_4,\mathbb{Z}}S(U_4\times U_2)$ is not contained in 
the $\mathrm{Index}_{\mathbb{Z}_4,\mathbb{Z}}\Omega $ it follows that 
\textbf{there is no equivariant map} $\Omega\rightarrow S(U_4\times U_2)$. This 
concludes the proof of Theorem \ref{Th:Main-2}.

\section{\label{Sec:Final}Concluding remarks}

\subsection{\label{Sec:Final-1} The $\mathbb{F}_2$-index}

Let $H^*(\mathbb{Z}_4,\mathbb{F}_2)=\mathbb{F}_2[e,u]/e^2$, $\deg (e)=1$, 
$\deg (u)=2$. The homomorphism of coefficients 
$j:\mathbb{Z}\rightarrow\mathbb{F}_2$, $j(1)=1$, 
induces a homomorphism in group cohomology 
$j^*:H^*(\mathbb{Z}_4;\mathbb{Z})\rightarrow H^*(\mathbb{Z}_4,\mathbb{F}_2)$ 
given by
$j^8(U)=u$ (compare \cite[Section 4.5.2]{B-Z}).

The $\mathbb{F}_2$-index of the configuration space $\Omega $ is
\begin{equation*}
\mathrm{Index}_{\mathbb{Z}_4,\mathbb{F}_2}\Omega =\langle eu^2,u^3\rangle .
\end{equation*}
This can be obtained in a similar fashion as we obtained the index with 
$\mathbb{Z}$-coefficients in Section \ref{Sec:SS}. The relevant $E_2$-term of
the Serre spectral sequence of the fibration $\Omega \rightarrow \Omega
\times_{\mathbb{Z}_4}\mathrm{E}\mathbb{Z}_4\rightarrow \mathrm{B}\mathbb{Z}_4 $ 
is described in Figure~\ref{Figure-4}.

\begin{figure}[tbh]
\centering\includegraphics[scale=0.60]{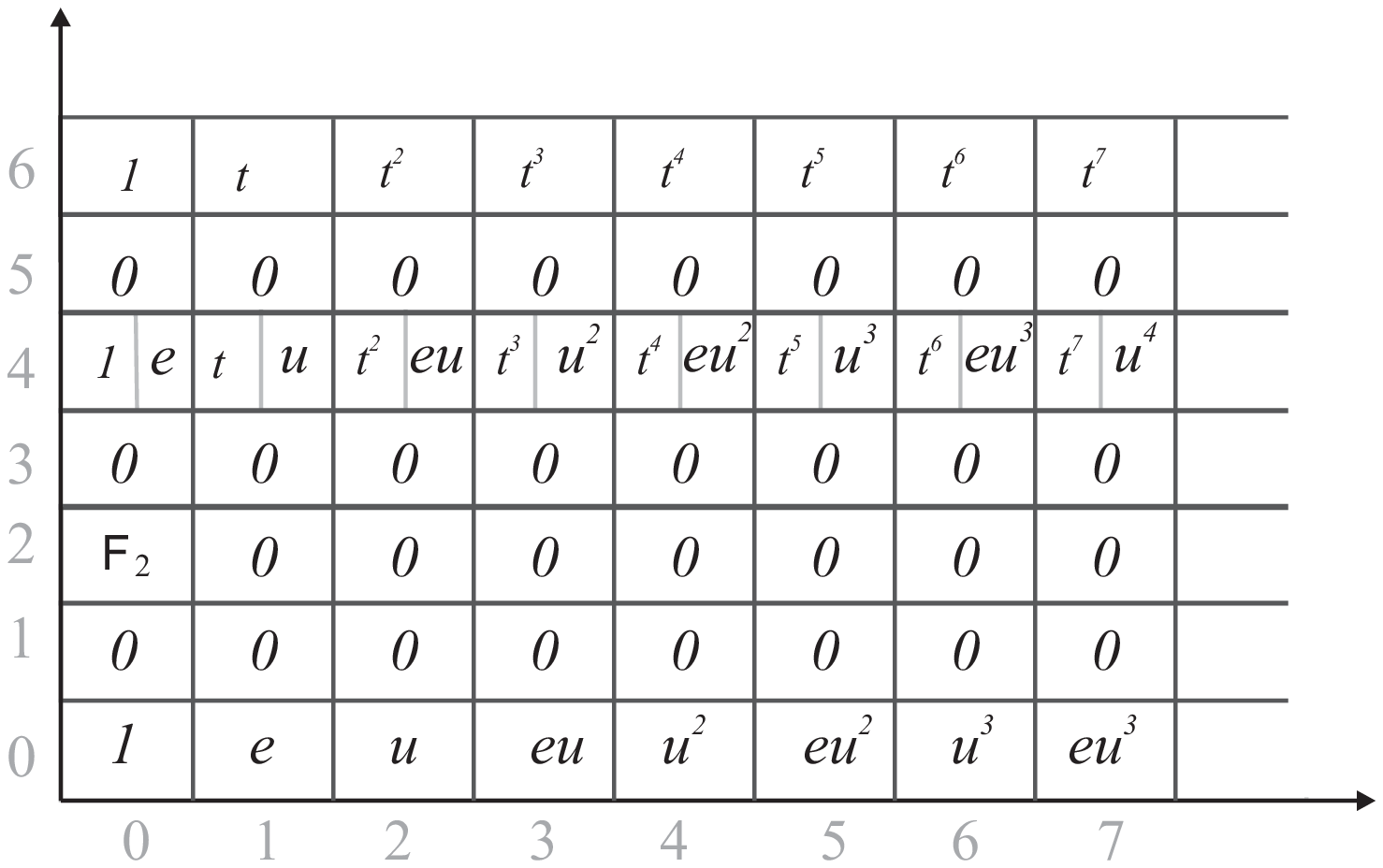}
\caption{$E_2$-term with $\mathbb{F}_2$-coefficients}
\label{Figure-4}
\end{figure}

The $\mathbb{F}_2$-index of the sphere $S(U_4\times U_2)$ is generated by
the $j^*$ image of the generator $2U^2$ of the index with 
$\mathbb{Z}$-coefficients 
$\mathrm{Index}_{\mathbb{Z}_4,\mathbb{Z}}S(U_4\times U_2)$.
Since $j^*(2U^2)=0$ the index 
$\mathrm{Index}_{\mathbb{Z}_4,\mathbb{F}_2}S(U_4\times U_2)$ is trivial. 
Therefore, for our problem \textbf{no
conclusion can be obtained from the study of the $\mathbb{F}_2$-index.} The
same observation holds even when the complete group $D_8$ is used. The 
$\mathbb{F}_2$-index of the sphere $S(U_4\times U_2)$ would be generated by 
$xyw=0\in H^*(D_8;\mathbb{F}_2)$, in the notation of \cite{B-Z}.

\subsection{\label{Sec:Final-2}The square peg problem}

The method of configuration spaces can also be set up for to the continuous
square peg problem. Following the ideas presented in Section \ref{Sec:CSTM},
taking for $X$ the product $S^1\times S^1\times S^1\times S^1$, for $Y$ the
subspace $Y=\left\{ (x,y,x,y)~|~x,y\in S^1\right\} $ and for the
configuration space $\Omega =X\backslash Y$, the square peg problem can be
related to the question of the existence of a $D_8$-equivariant map $\Omega
\rightarrow S(U_4\times U_2)$. The Fadell--Husseini indexes can be retrieved:%
\begin{equation*}
\mathrm{Index}_{\mathbb{Z}_4,\mathbb{Z}}\Omega =\langle U^2\rangle \text{
\qquad and \qquad } \mathrm{Index}_{\mathbb{Z}_4,\mathbb{Z}}S(U_4\times
U_2)=\langle 2U^2\rangle ,
\end{equation*}
but since 
$\mathrm{Index}_{\mathbb{Z}_4,\mathbb{Z}}\Omega \supseteq   
 \mathrm{Index}_{\mathbb{Z}_4,\mathbb{Z}}S(U_4\times U_2)$ 
the result does not yield
any conclusion. The same can be done for the complete symmetry group $D_8$,
explicitly 
$\mathrm{Index}_{D_8,\mathbb{Z}}S(U_4\times U_2)=\langle 2\mathcal{W}\rangle $
and 
$\mathcal{W}\in \mathrm{Index}_{D_8,\mathbb{Z}}\Omega $.

\medskip

\subsubsection*{Acknowledgements.}

Thanks to Anton Dochterman for many useful comments.


\begin{thebibliography}{9}
\itemsep=0pt

\bibitem{Aty} {\small \textsc{M. F. Atiyah,} \emph{Characters and cohomology
of finite groups}, Inst.\ Hautes \'{E}tudes Sci.\ Publ.\ Math.\ No.~9
(1961), 23--64.}

\bibitem{B-Z} {\small {\textsc{P. V. M. Blagojevi\'{c}, G. M. Ziegler,}}
\emph{The ideal-valued index for a dihedral group action, and mass partition
by two hyperplanes}, preprint, revised version \url{arXiv:0704.1943v2}, July
2008, 42 pages.}

\bibitem{Brown} {\small \textsc{K. S. Brown}, \emph{Cohomology of Groups},
Graduate Texts in Math. 87, Springer-Verlag, New York, Berlin, 1982.}

\bibitem{Bredon} {\small \textsc{G. E. Bredon}, \emph{Topology and Geometry}, 
Graduate Texts in Math. 139, Springer-Verlag, New York, 1993.}

\bibitem{Mun} {\small \textsc{J. R. Munkres}, \emph{Elements of Algebraic
Topology}, Addison-Wesley, Menlo Park CA, 1984.}

\bibitem{Shn} {\small \textsc{L. G. Shnirelman}, \emph{On certain
geometrical properties of closed curves (in Russian)}, Uspehi Matem. Nauk 10
(1944), 34--44, \url{http://tinyurl.com/28gsy3}.}

\bibitem{S} {\small \textsc{Y. Soibelman}, \emph{Topological Borsuk problem}, 
preprint \url{arXiv:math/0208221v2}, 2002, 4~pages.}
\end{thebibliography}
\end{document}